\documentclass[12pt]{article}  
\usepackage{amsmath}
\usepackage{amsfonts}
\usepackage{amssymb}
\usepackage{amsthm}
\usepackage{euscript}
\usepackage{ifthen}
\usepackage{xifthen}
\usepackage{xcolor}
\usepackage{fullpage}
\usepackage[english]{babel}
\usepackage{mathtools}
\usepackage{hyperref}
\usepackage{authblk}
\usepackage{color}

\usepackage{hyperref}
\hypersetup
{
	colorlinks,
	citecolor=black,
	filecolor=black,
	linkcolor=blue,
	urlcolor=blue
}
\hypersetup{linktocpage}
\newtheorem{theorem}{Theorem}[section]
\newtheorem{lemma}[theorem]{Lemma}
\newtheorem{corollary}[theorem]{Corollary}
\numberwithin{equation}{section}

\theoremstyle{definition}
 
\theoremstyle{remark}

\newcommand{\brac}[1]{\left(#1\right)}

\newcommand{\brab}[1]{\left\{#1\right\}}

\newcommand{\bk}{{\boldsymbol{k}}}

\newcommand{\br}{{\boldsymbol{r}}}
\newcommand{\bs}{{\boldsymbol{s}}}

\newcommand{\bx}{{\boldsymbol{x}}}

\newcommand{\bX}{{\boldsymbol{X}}}
\newcommand{\by}{{\boldsymbol{y}}}

\newcommand{\bW}{{\boldsymbol{W}}}

\newcommand{\rd}{{\rm d}}

\def\ZZd{{\mathbb Z}^d}

\def\RR{{\mathbb R}}
\def\RRd{{\mathbb R}^d}

\def\NN{{\mathbb N}}
\def\NNd{{\NN}^d}

\def\NN{{\mathbb N}}
\def\RR{{\mathbb R}}

\def\Vv{{\mathbb P}}

\def\Vv{{\mathcal P}}

\def\NNd{{\mathbb N}^d}
\def\RRd{{\mathbb R}^d}

\def\ZZd{{\mathbb Z}^d}

\def\Hh{{\mathcal H}}

\def\Pp{{\mathcal P}}

\def\Ss{{\mathcal S}}

\def\Vv{{\mathcal V}}

\def\NN{{\mathbb N}}
\def\RR{{\mathbb R}}

\def\NNd{{\mathbb N}^d}
\def\RRd{{\mathbb R}^d}

\def\Wap{W^r_p}

\newcommand{\norm}[2]{\left\|{#1}\right\|_{#2}}

\title{ \bf{Weighted  hyperbolic cross polynomial approximation}}

\author[a]{Dinh D\~ung}
\affil[a]{Information Technology Institute, Vietnam National University, Hanoi
	\protect\\
	144 Xuan Thuy, Cau Giay, Hanoi, Vietnam
	\protect\\
	Email: dinhzung@gmail.com
	\protect\\
	\bigskip
	\bigskip
{\it Dedicated to the 90th birthday of Professor Vladimir Tikhomirov}
	}

\date{\today}
\begin{document}
\maketitle

\begin{abstract}
 We study linear polynomial approximation  of   functions in weighted Sobolev spaces $W^r_{p,w}(\mathbb{R}^d)$ of mixed smoothness $r \in \mathbb{N}$, and their optimality in terms of Kolmogorov and linear  $n$-widths of the unit ball $\boldsymbol{W}^r_{p,w}(\mathbb{R}^d)$ in these spaces. The approximation error is measured by the norm of the weighted Lebesgue space $L_{q,w}(\mathbb{R}^d)$. The weight $w$ is a tensor-product  Freud  weight. 
For $1\le p,q \le \infty$ and $d=1$, we prove that the  polynomial approximation  by de la Vall\'ee Poussin  sums  of  the orthonormal polynomial expansion of functions with respect to the weight $w^2$, is asymptotically optimal in terms of  relevant linear $n$-widths  $\lambda_n\big(\bW^r_{p,w}(\RR), L_{q,w}(\RR)\big)$ and Kolmogorov $n$-widths  $d_n\big(\bW^r_{p,w}(\RR), L_{q,w}(\RR)\big)$ for $1\le q \le p <\infty$. 
For $1\le p,q \le \infty$ and $d\ge 2$, we  construct linear methods of hyperbolic cross polynomial approximation based on tensor product of  successive differences of dyadic-scaled de la Vall\'ee Poussin  sums,  which are counterparts  of  hyperbolic cross trigonometric linear polynomial approximation, and give some upper bounds of the error of these approximations  for various pair $p,q$ with $1 \le p, q \le \infty$. For  some particular weights $w$ and $d \ge 2$, we prove the  right convergence rate  of  $\lambda_n\big(\bW^r_{2,w}(\RRd), L_{2,w}(\RRd)\big)$ and $d_n\big(\bW^r_{2,w}(\RRd), L_{2,w}(\RRd)\big)$ which is performed by a constructive hyperbolic cross polynomial approximation.
 
 	\medskip
	\noindent
	{\bf Keywords and Phrases}:  Weighted approximation; Hyperbolic cross polynomial approximation;  De la Vall\'ee Poussin  sums; Kolmogorov widths; Linear widths; Weighted  Sobolev space of mixed smoothness;  Right convergence rate. 
	
	\medskip
	\noindent
	{\bf MSC (2020)}:    41A15; 41A25; 41A46; 41A63; 41A81.
	
\end{abstract}

\section{Introduction}
\label{Introduction}
We investigate weighted  linear hyperbolic cross polynomial approximations of functions  on $\RRd$ from weighted  Sobolev spaces of  mixed smoothness and their optimalities in terms of Kolmogorov and linear $n$-widths. 

We begin with a notion of weighted  Sobolev spaces of  mixed smoothness.  
Let 
\begin{equation} \nonumber
	w(\bx):= w_{\lambda,a,b}(\bx) := \bigotimes_{i=1}^d w(x_i), \ \ \bx \in \RRd,
\end{equation}
be the tensor product of $d$ copies of a generating univariate Freud  weight of the form
\begin{equation} \label{w(x)}
	w(x):= 	\exp \brac{- a|x|^\lambda + b},
	 \ \ \lambda > 1,  \ a >0, \ b \in \RR.
\end{equation} 
The most important parameter in the weight $w$ is $\lambda$. The parameter $b$ which produces only a positive constant in the weight $w$  is introduced for a certain normalization for instance, for the standard Gaussian weight which is one of the most important weights. 
In what follows,  we fix the  parameters $\lambda,a, b$ in the weight $w$.

Let  $1\leq p \le \infty$ and $\Omega$ be a Lebesgue measurable set on $\RRd$. 
We denote by  $L_{p,w}(\Omega)$ the weighted Lebesgue space  of all measurable functions $f$ on $\Omega$ such that the norm
\begin{align} \label{L-Omega}
\|f\|_{L_{p,w}(\Omega)} : = 
\begin{cases}
\bigg( \int_\Omega |f(\bx)w(\bx)|^p  \rd \bx\bigg)^{1/p},& \ \ 1\le p <  \infty; \\
\operatorname{ess \, sup}_{\bx \in \Omega}|f(\bx)w(\bx)|,& \ \ p = \infty,
\end{cases}
\end{align}
is finite.

 For $r \in \NN$ and $1 \le p \le \infty$, we define the weighted  Sobolev space $W^r_{p,w}(\Omega)$ of mixed smoothness $r$  as the normed space of all functions $f\in L_{p,w}(\Omega)$ such that the weak  partial derivative $D^{\bk} f$ belongs to $L_{p,w}(\Omega)$ for  every $\bk \in \NNd_0$ satisfying the inequality $|\bk|_\infty \le r$. The norm of a  function $f$ in this space is defined by
\begin{align} \nonumber
	\|f\|_{W^r_{p,w}(\Omega)}: = \Bigg(\sum_{|\bk|_\infty \le r} \|D^{\bk} f\|_{L_{p,w}(\Omega)}^p\Bigg)^{1/p}.
\end{align}

Let $\gamma$ be the standard $d$-dimensional Gaussian measure  with the density function 
$$
v_{\operatorname{g}}(\bx): = (2\pi)^{-d/2}\exp (- |\bx|_2^2/2).
$$
The well-known  spaces $L_p(\Omega;\gamma)$ and $\Wap(\Omega; \gamma)$ 
which are used in many applications, are defined  in the same way by replacing the norm \eqref{L-Omega} with the norm
$$
\|f\|_{L_p(\Omega; \gamma)} : = 
\bigg( \int_\Omega |f(\bx)|^p \gamma(\rd \bx)\bigg)^{1/p}
=
\bigg( \int_\Omega |f(\bx)\brac{v_{\operatorname{g}}}^{1/p}(\bx)|^p  \rd \bx\bigg)^{1/p}.
$$
Thus, the spaces $L_p(\Omega;\gamma)$ and $\Wap(\Omega; \gamma)$ coincide with  $L_{p,w}(\Omega)$ and $W^r_{p,w}(\Omega)$, where $w:= \brac{v_{\operatorname{g}}}^{1/p}$
for  a fixed $1 \le p < \infty$.  

Let $X$ be a Banach space and $F$ a central symmetric compact set in $X$. By linear approximation we understand an approximation of elements in $F$ by elements from a fixed finite-dimensional subspace $L$. For a given number $n \in \NN_0$, a natural question arising is how to choose  an optimal subspace of dimension at most $n$ for this approximation. This leads to the concept of the Kolmogorov $n$-width introduced  in 1936 \cite{Kol36}. 
The Kolmogorov $n$-width  of $F$ is defined by
\begin{equation*}
	d_n(F,X):= \inf_{L_{n}}\, \sup_{f\in F}\, \inf_{g\in L_n}\|f-g\|_X, 
\end{equation*}
where the left-most  infimum is  taken over all  subspaces $L_{n}$ of dimension  $\le n$ in $X$.

The Kolmogorov $n$-width provides a way to determine optimal approximation $n$-dimensional subspaces. Clearly, we would like to use as simple approximation operators as possible. In particular, the restriction by linear  operators leads to  the linear $n$-width of  $F$ in $X$ which was introduced by V.M. Tikhomirov \cite{Tikh1960} in 1960.
This $n$-width  is defined by
$$
\lambda_n(F,X):=\inf_{A_n} \sup_{f\in F} \|f-A_n(f)\|_X,
$$
where the infimum is taken over all linear operators $A_n$ in  $X$ with ${\rm rank}\, A_n\leq n$. In general, the Kolmogorov $n$-width and the linear $n$-width are different approximation characterizations. However,
if $X$ is a Hilbert space, then 
$\lambda_n(F,X) = d_n(F,X).$
In what follows, for a normed space $X$ of functions on $\Omega$, the boldface $\bX$ denotes the unit ball in $X$.

There is a large number of works devoted to the problem of (unweighted) linear hyperbolic cross approximations of functions having a mixed smoothness on a compact domain, and their optimalities in of terms Kolmogorov and linear $n$-widths,  see  for survey and bibliography in \cite{DTU18B}, \cite{NoWo10}, \cite{Tem18B}. Here by linear hyperbolic cross approximation we understand approximation of  multivariate periodic functions by trigonometric polynomials with frequencies from so-called hyperbolic crosses, or their counterpart for  multivariate non-periodic functions.

The    weighted polynomial approximation is a classical branch of approximation theory. There is a huge body of works on different aspects of  the  univariate weighted polynomial approximation. We refer the reader to the books \cite{Mha1996B}, \cite{Lu07B},  \cite{JMN2021} for relevant results  and bibliography.
In the recent paper \cite{DK2022}, we have studied the linear approximation of functions from $\Wap(\RRd; \gamma)$ with the error measured in $L_q(\RRd;\gamma)$ for $1 \le q \le p \le \infty$. In particular, we proved in the last paper  the right convergence rate 
\begin{align} \label{lambda_n-gamma}
	\lambda_n(\bW^r_2(\RRd; \gamma), L_2(\RRd;\gamma)) 
	=
	d_n(\bW^r_2(\RRd; \gamma), L_2(\RRd;\gamma)) 
	\asymp 
	n^{- r/2} (\log n)^{r(d-1)/2}.
\end{align}
In  the present paper, we  investigate linear hyperbolic cross polynomial approximation of functions  with a mixed smoothness on  $\RRd$. Functions to be approximated are in weighted Sobolev spaces $W^r_{p,w}(\RRd)$. The approximation error is measured by the norm of the weighted Lebesgue  spaces $L_{q,w}(\RRd)$. The values of $p,q$ may vary  satisfying the inequalities $1 \le p,q \le\infty$. 
The results on this approximation will imply upper bounds of  the high dimensional   linear $n$-widths  $\lambda_n\big(\bW^r_{p,w}(\RRd), L_{q,w}(\RRd)\big)$ and Kolmogorov $n$-widths  $d_n\big(\bW^r_{p,w}(\RRd), L_{q,w}(\RRd)\big)$ ($d \ge 2$), the right convergence rate of these  linear $n$-widths  in one-dimensional case ($d=1$).
 We also study   the right convergence rate  of 
 $\lambda_n\big(\bW^r_{2,w}(\RRd), L_{2,w}(\RRd)\big)$ and  $d_n\big(\bW^r_{2,w}(\RRd), L_{2,w}(\RRd)\big)$ for particular weights $w$.
 
We briefly describe the main results of the present paper.  Throughout the present paper, for given $1 \le p, q \le \infty$ and the parameter $\lambda > 1$ in the definition \eqref{w(x)} of the generating weight $w$,  we make use of  the notations
\begin{equation}\label{r_lambda}
r_\lambda:= (1 - 1/\lambda)r;
\end{equation}
\begin{equation}\nonumber
	\delta_{\lambda,p,q} 
	:=
	\begin{cases}
		(1-1/\lambda)(1/p - 1/q)&  \ \ \text{if} \ \ p \le q, \\
		(1/\lambda)(1/q - 1/p)&  \ \ \text{if} \ \ p > q;
	\end{cases}	 
\end{equation}
 and 
$$
r_{\lambda,p,q} := r_\lambda - 	\delta_{\lambda,p,q}.
$$ 

We also use  the abbreviations:
	$$
	d_n:= d_n(\bW^r_{p,w}(\RRd),  L_{q,w}(\RRd)), \ \ \lambda_n:= \lambda_n(\bW^r_{p,w}(\RRd),  L_{q,w}(\RRd)).
	$$
	Let $1\le p, q \le \infty$, $r_{\lambda,p,q} >0$ and $\Vv_{\xi}$ be the de la Vall\'ee Poussin hyperbolic cross sum operator (see \eqref{de la Valle Poussin Vv} for the definition).	Then we prove that 
	for $\xi >1$
	\begin{equation}\nonumber
		\big\|f - \Vv_\xi f\big\|_{L_{q,w}(\RRd)}  
		\ll 
		\|f\|_{W^r_{p,w}(\RRd)}
		\begin{cases}
			2^{- r_\lambda \xi} \xi^{d - 1}  &  \ \ \text{if} \ \ p = q, \\
			2^{- r_{\lambda,p,q}\xi} \xi^{(d - 1)/q} &  \ \ \text{if} \ \ p \not= q< \infty, \\
			2^{- r_{\lambda,p,q}\xi} \xi^{(d - 1)} &  \ \ \text{if} \ \  q=\infty,
		\end{cases}		
		\ \ \xi > 1, \  f \in W^r_{p,w}(\RRd). 
	\end{equation}	
If $\xi_n$ is the largest number such that $\operatorname{rank } \brac{\Vv_{\xi_n}} \le n$ for $n \in \NN$, 
as a consequence, we have that for $n \ge 2$,
\begin{equation}\label{UpperB-introduction}
	d_n \le \lambda_n 
	\le 
	\sup_{f\in \bW^r_{p,w}(\RRd)} 		\big\|f - \Vv_{\xi_n} f\big\|_{L_{q,w}(\RRd)}  
	\ll 
	\begin{cases}
		n^{-r_{\lambda}} (\log n)^{(r_{\lambda} + 1)(d-1)} &  \ \ \text{if} \ \ p = q, \\
		n^{-r_{\lambda,p,q}} (\log n)^{(r_{\lambda,p,q} + 1/q)(d-1)} &  \ \ \text{if} \ \ p \not= q<\infty,
	\\
	n^{-r_{\lambda,p,q}} (\log n)^{(r_{\lambda,p,q} + 1)(d-1)} &  \ \ \text{if} \ q=\infty.
	\end{cases}	 
\end{equation}

In  the one-dimensional  case when $d=1$, for $1 \le q \le p < \infty$ we prove the right convergence rate
\begin{equation}\nonumber
	d_n
	\asymp 
	\lambda_n 
	\asymp 
			\sup_{f\in \bW^r_{p,w}(\RR)} 	\big\|f - V_n f\big\|_{L_{q,w}(\RRd)}  
	\asymp		
	n^{-r_{\lambda,p,q}},
\end{equation}
where $V_nf$ is the $n$th de la Vall\'ee Poussin  sum of  the orthonormal polynomial expansion of $f$ with respect to   the multivariate  weight $w^2$.
	
The	linear polynomial approximation method  $\Vv_{\xi_n}$ performing the upper bounds \eqref{UpperB-introduction} -- a counterpart of hyperbolic cross trigonometric  approximation method --  is  based on tensor product of  successive differences of dyadic-scaled de la Vall\'ee Poussin sums of  the orthonormal polynomial expansion of $f$ with respect to   the multivariate  weight $w^2$.

For   $\lambda =2, 4$, we prove the right convergence rate for $n \ge 2$,
\begin{align*}
	\lambda_n(\bW^r_{2,w}(\RRd), L_{2,w}(\RRd)) 
	=
	d_n(\bW^r_{2,w}(\RRd), L_{2,w}(\RRd)) 
	\asymp 
	n^{- r_\lambda} (\log n)^{r_\lambda(d-1)},
\end{align*}		
which is a generalization of \eqref{lambda_n-gamma}.

The paper is organized as follows. In Section  \ref{Approximation by de la Vallee Poussin sums}, we study linear polynomial approximations in the norm $L_{q,w}(\RR)$ of  univariate functions from $\bW^r_{p,w}(\RR)$ by de la Vall\'ee Poussin  and Fourier sums of  the orthonormal polynomial expansion of functions with respect to the univariate weight $w^2$. We give some upper bounds of the error of these approximations for $1\le p,q \le \infty$, and prove their asymptotic optimality in terms of  linear $n$-widths  $\lambda_n\big(\bW^r_{p,w}(\RR), L_{q,w}(\RR)\big)$ and Kolmogorov $n$-widths  $d_n\big(\bW^r_{p,w}(\RR), L_{q,w}(\RR)\big)$ for $1\le q \le p <\infty$. 
In Section \ref{Hyperbolic cross approximation}, we study linear approximations of multivariate  functions $f \in \bW^r_{p,w}(\RRd)$. 
We construct linear methods of hyperbolic cross polynomial approximation. 
We give some upper bounds of the error of these approximations  for various pair $p,q$ with $1 \le p, q \le \infty$. In Section \ref{Asymptotic order of n-widths}, for the particular weights $w$ with $\lambda = 2,4$, we prove the  right convergence rate of   $n$-widths  $\lambda_n\big(\bW^r_{2,w}(\RRd), L_{2,w}(\RRd)\big)$ and   $d_n\big(\bW^r_{2,w}(\RRd), L_{2,w}(\RRd)\big)$.

\medskip
\noindent
{\bf Notation.} 
 Denote   $\bx=:\brac{x_1,...,x_d}$ for $\bx \in \RRd$; 
$|\bx|_p:= \brac{\sum_{j=1}^d |x_j|^p}^{1/p}$ $(1 \le p < \infty)$ and $|\bx|_\infty:= \max_{1\le j \le d} |x_j|$.  For $\bx, \by \in \RRd$, the inequality $\bx \le \by$ ($\bx < \by$) means $x_i \le y_i$ ( $x_i < y_i$) for every $i=1,...,d$.  We use letters $C$  and $K$ to denote general 
positive constants which may take different values. For the quantities $A_n(f,\bk)$ and $B_n(f,\bk)$ depending on 
$n \in \NN$, $f \in W$, $\bk \in \ZZd$,  
we write  $A_n(f,\bk) \ll B_n(f,\bk)$, $f \in W$, $\bk \in \ZZd$ ($n \in \NN$ is specially dropped),  
if there exists some constant $C >0$ independent of $n,f,\bk$  such that 
$A_n(f,\bk) \le CB_n(f,\bk)$ for all $n \in \NN$,  $f \in W$, $\bk \in \ZZd$ (the notation $A_n(f,\bk) \gg B_n(f,\bk)$ has the obvious opposite meaning), and  
$A_n(f,\bk) \asymp B_n(f,\bk)$ if $A_n(f,\bk) \ll B_n(f,\bk)$
and $B_n(f,\bk) \ll A_n(f,\bk)$.  Denote by $|G|$ the cardinality of the set $G$. 
For a Banach space $X$, denote by the boldface $\bX$ the unit ball in $X$.

	\section{Approximation by de la Vall\'ee Poussin sums}
\label{Approximation by de la Vallee Poussin sums}

In this section, we study linear approximations of univariate functions $f \in W^r_{p,w}(\RR)$ by de la Vall\'ee Poussin and Fourier sums of  the orthonormal polynomial expansion with respect to the univariate weight $w^2$. The approximation error is measured in the norm of  $L_{q,w}(\RR)$. We give some upper bounds of the error of these approximations and prove their asymptotic optimality in terms of  linear $n$-widths  $\lambda_n\big(\bW^r_{p,w}(\RR), L_{q,w}(\RR)\big)$ and Kolmogorov $n$-widths  $d_n\big(\bW^r_{p,w}(\RR), L_{q,w}(\RR)\big)$ for $1\le q \le p <\infty$.

Let $(p_k)_{k\in \NN_0}$ be the sequence of orthonormal polynomials with respect to the univariate weight 
\begin{equation} \label{w^2}
w^2(x)=\exp \brac{- 2a|x|^\lambda + 2b}. 
\end{equation}

The polynomials $\brac{p_k}_{k \in \NN_0}$ constitute an orthonormal basis of the Hilbert space $L_{2,w}(\RR)$, and 
every $f \in L_{2,w}(\RR)$ can be represented by the polynomial series 
\begin{equation}\nonumber
	f = \sum_{k \in \NN_0} \hat{f}(k) p_k \ \ {\rm with} \ \ \hat{f}(k) := \int_{\RR} f(x)\, p_k(x)w^2(x) \rd x
\end{equation}
converging in the norm of $L_{2,w}(\RR)$. Moreover,  there holds  Parseval's identity
\begin{equation}\nonumber
	\norm{f}{L_{2,w}(\RR)}^2= \sum_{k \in \NN_0} |\hat{f}(k)|^2.
\end{equation}

Since every polynomial belongs to the space $L_{q,w}(\RR)$ with $1\le q \le \infty$, we can define for any  $k \in \NN_0$, $m \in \NN$ and $f \in L_{p,w}(\RR)$
the $k$th Fourier coefficient
\begin{equation}\nonumber
\hat{f}(k) := \int_{\RR} f(x)\, p_k(x)w^2(x) \rd x;
\end{equation}
the $m$th Fourier sum
\begin{equation} \nonumber
	S_m f := \sum_{k = 0}^{m-1} \hat{f}(k) p_k; 
\end{equation} 
and the $m$th de la Vall\'ee Poussin sum
\begin{equation} \nonumber
	V_m := \frac{1}{m}\sum_{k = m+1}^{2m} S_k.
\end{equation} 

Let $\Pp_m$ denote the space of polynomials of degree at most $m$. From the definition we have the following properties of the operator $V_m$ for $1 \le p \le \infty$ and $m \in \NN$,
\begin{equation} \nonumber
	V_m f  \in \Pp_{2m-1}, \ \ f  \in L_{p,w}(\RR),
\end{equation} 
and
\begin{equation} \nonumber
	V_m \varphi = \varphi, \ \  \varphi \in \Pp_m.
\end{equation} 


 For $m \in \NN$, let $q_m$ be the Freud number defined by
\begin{equation}\nonumber
	q_m:= (m/a\lambda)^{1/\lambda} \asymp m^{1/\lambda},
\end{equation}
and $a_m$  the Mhaskar-Rakhmanov-Saff number defined by
\begin{equation}\nonumber
	a_m :=(\nu_\lambda m)^{1/\lambda} \asymp m^{1/\lambda}, \ \ \nu_\lambda:= \frac{2^{\lambda - 1} \Gamma(\lambda/2)^2}{\Gamma(\lambda)},
\end{equation}
and $\Gamma$ is the gamma function. From the definitions one can see that
\begin{equation}\label{q_m asymp a_m}
	q_m \asymp a_m \asymp m^{1/\lambda}.
\end{equation}
The numbers $q_m$ and $a_m$ are relevant to convergence rates of weighted polynomial approximation (see, e.g., \cite{Mha1996B,Lu07B}).

For $1 \le p \le \infty$ and $f \in L_{p,w}(\RR)$, we define 	
\begin{equation}\nonumber
	E_m(f)_{p,w}:= \inf_{\varphi \in \Pp_m}	\|f - \varphi\|_{L_{p,w}(\RR)}  
\end{equation}
as the quantity of best approximation of $f$ by polynomials of degree at most $m$.
Then there holds the inequalities  for $1 \le p \le \infty$ 
\cite[Proposition 4.1.2, Lemma 4.1.5]{Mha1996B}
\begin{equation}\label{|V_m f|<}
	\|V_m f\|_{L_{p,w}(\RR)}
	\ll
		\|f\|_{L_{p,w}(\RR)}, \ \  f \in L_{p,w}(\RR), 
\end{equation}
\begin{equation}\label{E_{2m}(f)_{p,w}}
	E_{2m}(f)_{p,w} 
	\le 
	\|f  - V_m f\|_{L_{p,w}(\RR)}
	\ll
	E_m(f)_{p,w}, \ \  f \in L_{p,w}(\RR),  
\end{equation}
and \cite[Theorem 4.1.1]{Mha1996B} taking account \eqref{q_m asymp a_m}
\begin{equation}\label{	E_m(f)_{p,w}<}
	E_m(f)_{p,w} 
	\le 
	m^{-r_\lambda}  \|f \|_{W^r_{p,w}(\RR)}, \ \  f \in W^r_{p,w}(\RR).  
\end{equation}
From  \eqref{q_m asymp a_m}, \eqref{E_{2m}(f)_{p,w}} and \eqref{	E_m(f)_{p,w}<} it follows that
if $1 \le p \le \infty$,	then  for every $m \in \NN$,
	\begin{equation}\label{|f - I_m f|<}
		\|f - V_m f\|_{L_{p,w}(\RR)} 
		\le
		C m^{- r_\lambda} \|f \|_{W^r_{p,w}(\RR)}, \ \ f \in W^r_{p,w}(\RR).  
	\end{equation}

For the operators $S_m$, we have  \cite{JL1995}
\begin{equation}\label{|S_m f|<}
	\|S_m f\|_{L_{p,w}(\RR)}
	\ll
	\|f\|_{L_{p,w}(\RR)}, \ \  f \in L_{p,w}(\RR), \ \ \text{if  and  only if} \ \ 4/3 < p < 4,
\end{equation}
or, equivalently,
\begin{equation}\label{f - S_m f}
	\|f - S_m f\|_{L_{p,w}(\RR)}
	\ll
	E_m(f)_{p,w}, \ \  f \in L_{p,w}(\RR), \ \ \text{if  and  only  if} \ \ 4/3 < p < 4.
\end{equation}

For proofs of the following  lemmata  see \cite[ Theorem 3.4.2,  Theorem 4.2.4]{Mha1996B}, providing \eqref{q_m asymp a_m}.
\begin{lemma} \label{lemma:B-NInequality}
	Let $1 \le  p,q \le \infty$. Then we have the following.
\begin{itemize}
	\item[\rm{(i)}] 
	There holds the Markov-Bernstein-type inequality  
	\begin{equation}\nonumber
		\|\varphi'\|_{L_{p,w}(\RR)} 
		\ll
		m^{1- 1/\lambda}\|\varphi\|_{L_{p,w}(\RR)}	 	 
		\ \ \ \
		\forall  \varphi \in \Pp_m, \ \ \forall m \in \NN.  
	\end{equation}
		\item[\rm{(ii)}] For $1 \le p < q \le \infty$, there holds the Nikol'skii-type inequality
			\begin{equation}\nonumber
			\|\varphi\|_{L_{q,w}(\RR)} 
			\ll
			m^{(1-1/\lambda)(1/p - 1/q)}\|\varphi\|_{L_{p,w}(\RR)}	 	 
			\ \ \ \
			\forall  \varphi \in \Pp_m, \ \ \forall m \in \NN.  
		\end{equation}		
	\item[\rm{(iii)}] For $1 \le q < p \le \infty$, there holds the Nikol'skii-type inequality
\begin{equation}\nonumber
	\|\varphi\|_{L_{q,w}(\RR)} 
	\ll
	m^{(1/\lambda)(1/q - 1/p)}\|\varphi\|_{L_{p,w}(\RR)},	 	 
	 \ \
	  \forall \varphi \in \Pp_m, \  \forall m \in \NN.  
\end{equation}		
\end{itemize}	
\end{lemma}

We define  the one-dimensional operators for  $m \in \NN$ and $k \in \NN_0$
\begin{equation} \label{v_k}
	v_{m,k}:=  V_{m2^k} - V_{m2^{k-1}}, \  k >0, \ \ v_{m,0}:= V_m,
\end{equation}
and
\begin{equation}\nonumber
	s_{m,k}:=  S_{m2^k} - S_{m2^{k-1}}, \  k >0, \ \ s_{m,0}:= S_m. 
\end{equation}
We also use the abbreviations: $v_k := v_{1,k}$ and $s_k := s_{1,k}$.
\begin{lemma} \label{lemma:v_k}
	Let  $1\le p, q \le \infty$, $r_{\lambda,p,q} >0$ and $m \in \NN$.   Then we have that
	for every $ f \in W^r_{p,w}(\RR)$, there hold the series representation
	\begin{equation}\label{Series2}
		f
		= 	\sum_{k \in \NN_0}v_{m,k} f 
	\end{equation}
	with absolute convergence in the space $L_{q,w}(\RR)$  of the series, and the norm estimates 
	\begin{equation}\label{Delta_k^If}
		\big\|v_{m,k} f\big\|_{L_{q,w}(\RR)}  
		\ll (m2^k)^{-  r_{\lambda,p,q}} \|f\|_{W^r_{p,w}(\RR)},  \ \ f \in W^r_{p,w}(\RR),  \ \ m \in \NN, 
		\   k \in \NN_0.
	\end{equation}
\end{lemma}

\begin{proof}
Let 	$ f \in W^r_{p,w}(\RR)$.
Since $v_{m,k} f \in \Pp_{m2^{k +1} -1}$ by the claims (iii) and (iv) of  Lemma~\ref{lemma:B-NInequality} we have that
		\begin{equation}\label{Delta_k^If<<}
		\big\|v_{m,k} f\big\|_{L_{q,w}(\RR)}  
		\ll (m2^k)^{\delta_{\lambda,p,q}} 	\big\|v_{m,k} f\big\|_{L_{p,w}(\RR)},
		\ \ m \in \NN, \   k \in \NN_0.
	\end{equation}	
	By Lemma \ref{lemma:v_k} we have that for every $f \in W^r_{p,w}(\RR)$ and $k \in \NN_0$,
	\begin{equation}\nonumber
		\begin{aligned}
			\big\|v_{m,k} f\big\|_{L_{p,w}(\RR)} 
			& \le  \big\|f  - V_{m2^k}\big\|_{L_{p,w}(\RR)}  + \big\|f  - V_{m2^{k-1}}\big\|_{L_{p,w}(\RR)} 
			\\
			&\ll 
			  (m2^k)^{- r_\lambda k} \|f\|_{W^r_{p,w}(\RR)},
		\end{aligned}
	\end{equation}
	which together with \eqref{Delta_k^If<<} proves \eqref{Delta_k^If} and hence  the absolute convergence of the series in \eqref{Series2} follows. 
	The equality in \eqref{Series2} is implied from \eqref{|f - I_m f|<}  and the equality 
	\begin{equation}\nonumber
		V_{m2^k}
		= 	\sum_{s \le k} v_{m,s}.
	\end{equation}	
	\hfill	
\end{proof}	

\begin{theorem} \label{thm: V_n}
	Let $1\le p,q \le \infty$ and $r_{\lambda,p,q} >0$.     Then we have
		\begin{equation}\nonumber
				\sup_{f\in \bW^r_{p,w}(\RR)} \big\|f - V_n f\big\|_{ L_{q,w}(\RR)} 
			\ll
			n^{- r_{\lambda,p,q}}.
		\end{equation}
\end{theorem}
\begin{proof}
By using Lemma \ref{lemma:v_k} we derive  for every $f\in \bW^r_{p,w}(\RR)$ and $n \in \NN_0$,
\begin{equation}\label{}
	\begin{split}\nonumber
 \big\|f - V_n f\big\|_{ L_{q,w}(\RR)} 
		&= 
		\Bigg\|\sum_{k \in \NN} v_{n,k} f\Bigg\|_{ L_{q,w}(\RR)} \\
		&\le 
		\sum_{k \in \NN} \big\|v_{n,k} f\big\|_{ L_{q,w}(\RR)}
		\ll
		\sum_{k \in \NN}  (n2^k)^{-  r_{\lambda,p,q}}  \|f\|_{W^r_{p,w}(\RRd)} \\
		&\le
		n^{-  r_{\lambda,p,q}} \sum_{k \in \NN}  2^{-  r_{\lambda,p,q} k}
		\asymp
		n^{- r_{\lambda,p,q}}. 
	\end{split}
\end{equation}
	\hfill
\end{proof}
\begin{corollary} \label{cor: d_n}
	Let  $1\le q \le p < \infty$ and $r_{\lambda,p,q} >0$.     Then we have
	\begin{equation}\label{d_n}
\lambda_n\big(\bW^r_{p,w}(\RR),  L_{q,w}(\RR)\big) 	
\asymp
d_n\big(\bW^r_{p,w}(\RR),  L_{q,w}(\RR)\big) 	
\asymp				
		n^{- r_{\lambda,p,q}}.
	\end{equation}
\end{corollary}
\begin{proof}
	The upper bound of \eqref{d_n} can be easily derived from Theorem \ref{thm: V_n}. The lower bound was proven in  \cite[(2.32)]{DD2024}.
	\hfill
\end{proof}

Similarly, from \eqref{f - S_m f} and \eqref{	E_m(f)_{p,w}<}  we deduce the following results for the approximation by Fourier sums.

\begin{lemma} \label{lemma:s_k}
	Let  $4/3 <  p < 4$, $1\le q \le \infty$, $r_{\lambda,p,q} >0$ and $m \in \NN$.   Then we have that
	for every $ f \in W^r_{p,w}(\RR)$, there hold the series representation
	\begin{equation}\nonumber
		f
		= 	\sum_{k \in \NN_0}s_{m,k} f 
	\end{equation}
	with absolute convergence in the space $L_{q,w}(\RR)$  of the series, and the norm estimates 
	\begin{equation}\label{Delta_k^If}
		\big\|s_{m,k} f\big\|_{L_{q,w}(\RR)}  
		\leq C (m2^k)^{-  r_{\lambda,p,q}} \|f\|_{W^r_{p,w}(\RR)}, \ \ m \in \NN, 
		\   k \in \NN_0.
	\end{equation}
\end{lemma}

\begin{theorem} \label{thm: S_n}
	Let $4/3 <  p < 4$, $1\le q \le \infty$ and $r_{\lambda,p,q} >0$.     Then we have
	\begin{equation}\nonumber
		\sup_{f\in \bW^r_{p,w}(\RR)} \big\|f - S_n f\big\|_{ L_{q,w}(\RR)} 
		\ll
		n^{- r_{\lambda,p,q}}.
	\end{equation}
\end{theorem}

	\section{Hyperbolic cross polynomial approximation}
\label{Hyperbolic cross approximation}

In this section, we consider weighted  hyperbolic cross   linear polynomial approximations of multivariate  functions $f \in \bW^r_{p,w}(\RRd)$. 
The approximation error is measured in the norm of  $L_{q,w}(\RRd)$. We construct linear methods of polynomial approximation which are counterparts  of linear  hyperbolic cross trigonometric approximation for periodic multivariate functions. 
We give some upper bounds of the error of these approximations and  of  linear $n$-widths  $\lambda_n\big(\bW^r_{p,w}(\RRd), L_{q,w}(\RRd)\big)$ and Kolmogorov $n$-widths  $d_n\big(\bW^r_{p,w}(\RRd), L_{q,w}(\RRd)\big)$ for various pair $p,q$ with $1 \le p, q \le \infty$. For the weights $w$ with $\lambda = 2,4$, we establish the right convergence rate of   $n$-widths  $\lambda_n\big(\bW^r_{2,w}(\RRd), L_{2,w}(\RRd)\big)$ and   $d_n\big(\bW^r_{2,w}(\RRd), L_{2,w}(\RRd)\big)$.

Recall that  $(p_k)_{k \in \NN_0}$ is the sequence of orthonormal  polynomials with respect to the univariate  Freud-type weight $w^2$ as in \eqref{w^2}.
For every multi-index $\bk\in \NNd_0$, the $d$-variate 
polynomial $p_\bk$, we  define
\begin{equation*}\label{H_bk}
	p_\bk(\bx) :=\prod_{j=1}^d p_{k_j}(x_j),
	\ \ \bx\in \RRd.
\end{equation*}
The polynomials $\brac{p_\bk}_{\bk \in \NNd_0}$ constitute an orthonormal basis of the Hilbert space $L_{2,w}(\RRd)$, and 
every $f \in L_{2,w}(\RRd)$ can be represented by the polynomial series 
\begin{equation}\label{ONP-series}
	f = \sum_{\bk \in \NNd_0} \hat{f}(\bk) p_\bk \ \ {\rm with} \ \ \hat{f}(\bk) := \int_{\RRd} f(\bx)\, p_\bk(\bx)w(\bx) \rd \bx
\end{equation}
converging in the norm of $L_{2,w}(\RRd)$. Moreover,  there holds  Parseval's identity
\begin{equation}\label{P-id}
	\norm{f}{L_{2,w}(\RRd)}^2= \sum_{\bk \in \NNd_0} |\hat{f}(\bk)|^2.
\end{equation}

For $\bx \in \RRd$ and $e \subset \brab{1,...,d}$, let $\bx^e \in \RR^{|e|}$ be defined by $(x^e)_i := x_i$, and  $\bar{\bx}^e\in \RR^{d-|e|}$ by $(\bar{x}^e)_i := x_i$, $i \in \brab{1,...,d} \setminus e$. With an abuse we write 
$(\bx^e,\bar{\bx}^e) = \bx$.

For the proof of the following lemma, see \cite[Lemma 3.2]{DD2023}. 

\begin{lemma} \label{lemma:g(bx^e}
	Let $1\le p \le \infty$,  $e \subset \brab{1,...,d}$ and $\br \in \NNd_0$. Assume that $f$ is a function on $\RRd$  such that for every $\bk \le \br$, $D^\bk f \in L_{p,w}(\RRd)$. 
	Put for  $\bk \le \br$ and $\bar{\bx}^e \in \RR^{d-|e|}$,
	\begin{equation*}\label{g(bx^e}
	g(\bx^e): =  D^{\bar{\bk}^e} f(\bx^e,\bar{\bx}^e).
	\end{equation*}
Then $D^\bs g \in L_{p,w}(\RR^{|e|})$ for every $\bs \le \bk^{e}$ and almost every 
$\bar{\bx}^e \in \RR^{d-|e|}$.
\end{lemma}

Based on the operators $v_k:= v_{1,k} \ k \in \NN_0$, defined in \eqref{v_k}, we construct approximation operators for functions in $L_{p,w}(\RRd)$  by using the well-known Smolyak algorithm. 
We  define for $k \in \NN_0$, the one-dimensional operators
\begin{equation*}\label{E}
	E_k f:= f - V_{2^k} f, \ \ \ k \in \NN_0.
\end{equation*}
For $\bk \in \NNd$, the $d$-dimensional operators $V_{2^\bk}$,  $v_\bk$ and $E_\bk$ are defined as the tensor  product of one-dimensional operators:
\begin{equation}\nonumber
	V_{2^\bk}:= \bigotimes_{i=1}^d V_{2^{k_i}} , \ \	
	v_\bk:= \bigotimes_{i=1}^d v_{k_i}, \ \ 	
	E_\bk:= \bigotimes_{i=1}^d E_{k_i}, 
\end{equation}
where $2^\bk:= (2^{k_1},\cdots, 2^{k_d})$ and 
the univariate operators $V_{2^{k_j}}$, $v_{k_j}$ and $E_{k_j}$ 
 are successively applied to the univariate functions $\bigotimes_{i<j} V_{2^{k_i}}(f)$, $\bigotimes_{i<j} v_{k_i}(f)$ and $\bigotimes_{i<j} E_{k_i} $, respectively, by considering them  as 
functions of  variable $x_j$ with the other variables held fixed. The operators $V_{2^\bk}$, $v_\bk$ and $E_\bk$ are well-defined for  functions  from $L_{p,w}(\RRd)$ for $1 \le p \le \infty$.

 Observe  that 
\begin{equation}\nonumber
v_\bk f =  \sum_{e \subset \brab{1,...,d}} (-1)^{d - |e|}V_{2^{\bk(e)}} f, 
\end{equation}
where  $\bk(e) \in \NNd_0$ is defined by $k(e)_i = k_i$, $i \in e$, and 	$k(e)_i = \max(k_i-1,0)$, $i \not\in e$.  We also have
\begin{equation}\nonumber
	\big(E_\bk f \big)(\bx)=  \sum_{e \subset \brab{1,...,d}} (-1)^{|e|} 
\big(V_{2^{\bk^e}}f(\cdot,\bar{\bx^e})\big)(\bx^e). 
\end{equation}

\begin{lemma} \label{lemma:E_k}
		Let $1\le p, q \le \infty$  and $r_{\lambda,p,q} >0$.    Then we have that
	\begin{equation*}\label{E_k}
	\big\|E_\bk f\big\|_{L_{q,w}(\RRd)}  
		\leq 2^{-r_{\lambda,p,q}|\bk|_1} \|f\|_{W^r_{p,w}(\RRd)}, 
		\ \  \bk \in \NNd_0, \ \ f \in W^r_{p,w}(\RRd).
	\end{equation*}
\end{lemma}

\begin{proof} The case $d=1$ of the lemma follows from Theorem \ref{thm: V_n}.  For simplicity we prove the lemma for the case $d=2$ and $q<\infty$. The general case can be proven in the same way by induction on $d$.
Indeed, by  applying successively the case  $d=1$ of the lemma with respect to variables $x_2$ and $x_1$ we obtain 	
		\begin{equation}\nonumber
		\begin{aligned}
			\big\|E_{(k_1,k_2)}f\big\|_{L_{q,w}(\RR^2)}^q
			& =
			\int_{\RR}\int_{\RR}\big|E_{k_2} (E_{k_1}f(x_1,x_2))\big|^q w(\bx)^q\rd x_2 \rd x_1
			\\
			&\leq 
			2^{-q r_{\lambda,p,q} k_2}	\int_{\RR} \sum_{s_2=0}^r\int_{\RR}
			\big|D^{(0,s_2)} (E_{k_1}f(x_1,x_2))\big|^q w(\bx)^q\rd x_2 \rd x_1
             \\
			&= 
				2^{-q r_{\lambda,p,q} k_2}	\int_{\RR} \sum_{s_2=0}^r\int_{\RR}
			\big| (E_{k_1}D^{(0,s_2)}f(x_1,x_2))\big|^q w(\bx)^q\rd x_2 \rd x_1
		   \\
		&\le
			2^{-q r_{\lambda,p,q} k_2}	\int_{\RR} \sum_{s_2=0}^r\int_{\RR}
			2^{-q r_{\lambda,p,q} k_1}	\sum_{s_1=0}^r\int_{\RR}
		\big| D^{(s_1,s_2)}f(x_1,x_2)\big|^q w(\bx)^q\rd x_1 \rd x_2
		      \\
		   &= 
		   2^{-q r_{\lambda,p,q} |\bk|_1}\sum_{|\bs|_\infty \le r}\int_{\RR^2}
		   \big| D^{(s)}f(\bx)\big|^q w(\bx)^q\rd \bx
 \\
&= 
2^{- q r_{\lambda,p,q}|\bk|_1} \|f\|_{W^r_{p,w}(\RR^2)}^q.   
		\end{aligned}
	\end{equation}	
\hfill
\end{proof}

We say that $\bk \to \infty$, $\bk \in \NNd_0$, if and only if $k_i \to \infty$ for every $i = 1,...,d$.
\begin{lemma} \label{lemma:v_bk}
	Let $1\le p, q \le \infty$  and $r_{\lambda,p,q} >0$.    Then we have that
	for every $ f \in W^r_{p,w}(\RRd)$,
		\begin{equation}\label{Series1}
f
	= 	\sum_{\bk \in \NNd_0}v_\bk f 
	\end{equation}
with absolute convergence in the space $L_{q,w}(\RRd)$  of the series, and
	\begin{equation}\label{Delta_k}
			\big\|v_\bk f\big\|_{L_{q,w}(\RRd)}  
		\leq C 2^{-r_{\lambda,p,q}|\bk|_1} \|f\|_{W^r_{p,w}(\RRd)}, 
		\ \  \bk \in \NNd_0.
	\end{equation}
\end{lemma}

\begin{proof}  The operator $v_\bk$ can be represented in the form
	\begin{equation}\nonumber
	v_\bk f
	= 	\sum_{e \subset \brab{1,...,d}} (-1)^{|e|} E_{\bk(e)} f.
\end{equation}	
Therefore, by using Lemma \ref{lemma:E_k} we derive that for every $f \in W^r_{p,w}(\RRd)$ and  $\bk \in \NNd_0$,
	\begin{equation}\nonumber
		\begin{aligned}
			\big\|v_\bk f\big\|_{L_{q,w}(\RRd)} 
			& \le
			\sum_{e \subset \brab{1,...,d}} \big\|E_{\bk(e)} f\big\|_{L_{q,w}(\RRd)}
			\\
			&\leq 
			\sum_{e \subset \brab{1,...,d}}  C 2^{-r_{\lambda,p,q}|\bk(e)|_1} \|f\|_{W^r_{p,w}(\RRd)}
			\le C 2^{-r_{\lambda,p,q}|\bk|_1} \|f\|_{W^r_{p,w}(\RRd)},
		\end{aligned}
	\end{equation}
which proves \eqref{Delta_k} and hence  the absolute convergence of the series in \eqref{Series1} follows.
 Notice that 
\begin{equation}\nonumber
f - V_{2^\bk}f
	= 	\sum_{e \subset \brab{1,...,d}, \ e \not= \varnothing} (-1)^{|e|} E_{\bk^e}f,
\end{equation}	
where recall $\bk^e \in \NNd_0$ is defined by $k^e_i = k_i$, $i \in e$, and 	$k^e_i = 0$, $i \not\in e$.
By using Lemma~\ref{lemma:E_k} we derive  for $\bk \in \NNd_0$ and $f \in W^r_{p,w}(\RRd)$,
\begin{equation}\nonumber
	\begin{aligned}
		\big\|f - V_{2^\bk}f\big\|_{L_{q,w}(\RRd)}  
		& \le
		\sum_{e \subset \brab{1,...,d}, \ e \not= \varnothing} \big\|E_{\bk^e} f\big\|_{L_{q,w}(\RRd)}
		\\
		&\leq 
		C \max_{e \subset \brab{1,...,d}, \ e \not= \varnothing} \ \max_{1 \le i \le d} 2^{-r_{\lambda,p,q}|k^e_i|} \|f\|_{W^r_{p,w}(\RRd)}
	\\
	&\leq 
	C \max_{1 \le i \le d} 2^{-r_{\lambda,p,q}|k_i|} \|f\|_{W^r_{p,w}(\RRd)},
	\end{aligned}
\end{equation}
which is going to $0$ when $\bk \to \infty$. This together with the obvious equality
	\begin{equation}\nonumber
V_{2^\bk} f
	= 	\sum_{\bk \in \NNd_0: \,  \bs \le \bk} v_{\bs}f
\end{equation}	
proves  \eqref{Series1}.
	\hfill
\end{proof}

 For $\xi > 0$, we define the de la Vall\'ee Poussin hyperbolic cross linear operator $\Vv_\xi$ for functions $f \in L_{q,w}(\RRd)$ by
	\begin{equation}\label{de la Valle Poussin Vv}
	\Vv_\xi f
	:= 	\sum_{\bk \in \NNd_0: \, |\bk|_1 \le \xi } v_{\bk} f.
\end{equation}	
Notice that the function $\Vv_\xi f$ belongs to the polynomial subspace 
	\begin{equation}\nonumber
	\Pp(\xi) 
	:= 	\operatorname{span}\brab{p_\bs: \,  \bs \in H(\xi)},
\end{equation}	
where
	\begin{equation}\nonumber
	H(\xi) 
	:= 	\bigcup_{\bk \in \NNd_0: \, |\bk|_1 \le \xi }\brab{\bs \in \NNd_0: \,  \bs < 2. 2^{2\bk} }.
\end{equation}	
From \eqref{|V_m f|<} it follows  that $\Vv_{\xi}$ is a linear bounded operator  in $L_{q,w}(\RRd)$ for $1 \le q \le \infty$, and
\begin{equation}\label{rank Vv}
	\operatorname{rank } \brac{\Vv_{\xi}} = |H(\xi) | 
	= \sum_{|\bk|_1 \le \xi} \prod_{j=1}^d (2^{k_j +1} - 1)
	\asymp 
	2^{\xi} \xi^{d - 1}.
\end{equation}

The multi-index set $H(\xi)$ consists of  the non-negative elements of the step hyperbolic cross
	\begin{equation}\nonumber
	\tilde{H}(\xi) 
	:= 	\bigcup_{\bk \in \NNd_0: \, |\bk|_1 \le \xi }\brab{\bs \in \ZZd: \,  |s_i| < 2^{2k_i}, \, i = 1,...,d },
\end{equation}	
which is similar by the form  to  the frequency set of  trigonometric polynomials used in the classical hyperbolic cross approximation (see \cite{DTU18B} for details). Hence with an abuse, we call an approximation by elements from subspaces $H(\xi)$ weighted hyperbolic cross polynomial approximation, and $\Vv_{\xi}f$ de la Vall\'ee Poussin hyperbolic cross sum of  the orthonormal polynomial expansion of $f$ with respect to   the multivariate  weight $w^2$.

In what follows, for short we write  $|\bk|_1 \le \xi$ as $\bk \in \NNd_0: \, |\bk|_1 \le \xi$ and etc., if there is not misunderstanding. Let 
$$\Pp_{2^\bk}:= 	\operatorname{span}\brab{p_\bs: \,  s_i \le 2^{k_i}, \, i = 1,...,d}.$$

For the proof of the following lemma, see \cite[Lemma A.2]{DD2024}.

\begin{lemma}  \label{lemma:IneqL_qNorm<L_pNorm}
	Let $1 \le  p, q < \infty$, $p \not= q$ and  
	$f \in L_{q,w}(\RRd)$ be represented by  the  series 
	\begin{equation*} 
		f \ = \sum_{\bk \in \NNd_0} \ \varphi_\bk, \ \varphi_\bk \in \Pp_{2^\bk},
	\end{equation*}
	converging in $L_{q,w}(\RRd)$.
	Then there holds the inequality 
	\begin{equation} \nonumber
		\|f\|_{L_{q,w}(\RRd)}
		\ \le C \left( \sum_{\bk \in \NNd_0} \ \|2^{\delta_{\lambda,p,q}|\bk|_1}  \varphi_\bk \|_{L_{p,w}(\RRd)}^q \right)^{1/q}, 
	\end{equation}
	with some constant $C$ depending at most on $\lambda,p,q,d$, whenever the right side is finite. 
\end{lemma}


\begin{theorem} \label{thm:Vv_xi-error}
	Let $1\le p, q \le \infty$ and $r_{\lambda,p,q} >0$.    Then we have that	for $\xi >1$,
\begin{equation}\label{upperbound}
	\big\|f - \Vv_\xi f\big\|_{L_{q,w}(\RRd)}  
	\ll 
	 \|f\|_{W^r_{p,w}(\RRd)}
\begin{cases}
	2^{- r_\lambda \xi} \xi^{d - 1}  &  \ \ \text{if} \ \ p = q, \\
	2^{- r_{\lambda,p,q}\xi} \xi^{(d - 1)/q} &  \ \ \text{if} \ \ p \not= q< \infty, \\
2^{- r_{\lambda,p,q}\xi} \xi^{d - 1} &  \ \ \text{if} \ \  q=\infty,
\end{cases}		
	\ \ \xi > 1, \  f \in W^r_{p,w}(\RRd). 
\end{equation}
\end{theorem}

\begin{proof}  From  Lemma \ref{lemma:v_bk} we derive that for $\xi >1$ and $f \in W^r_{p,w}(\RRd)$,
		\begin{equation}\label{Series1}
	f - \Vv_\xi f 
	= 	\sum_{|\bk|_1 > \xi } v_\bk f, \ \  v_\bk f \in \Pp_{2^{\bk}},
\end{equation}
with absolute convergence in the space $L_{q,w}(\RRd)$  of the series, and there holds  \eqref{Delta_k}. If $p \not= q$, applying Lemma	\ref{lemma:IneqL_qNorm<L_pNorm} and \eqref{Delta_k}, we obtain \eqref{upperbound}:
	\begin{equation}\nonumber
		\begin{aligned}
		\big\|f - \Vv_\xi f \big\|_{L_{q,w}(\RRd)}^q  
			& \ll
				\sum_{|\bk|_1 > \xi } \big\|2^{\delta_{\lambda,p,q}|\bk|_1}v_\bk f \big\|_{L_{p,w}(\RRd)}^q
				\ll \sum_{|\bk|_1 > \xi } 2^{-q r_{\lambda,p,q}|\bk|_1}  \|f\|_{W^r_{p,w}(\RRd)}^q
		\\
		&= 
	 \|f\|_{W^r_{p,w}(\RRd)}^q \sum_{|\bk|_1 > \xi } 2^{-q r_{\lambda,p,q}|\bk|_1} 
	\ll  2^{- q r_{\lambda,p,q}\xi} \xi^{d - 1} \|f\|_{W^r_{p,w}(\RRd)}^q.
		\end{aligned}
	\end{equation}
	
	If $p  = q$  or $q=\infty $, the upper bound \eqref{upperbound} can be derived similarly by using \eqref{Series1}, \eqref{Delta_k} and the inequality
	\begin{equation}\nonumber
			\big\|f - \Vv_\xi f \big\|_{L_{q,w}(\RRd)} 
			\le
			\sum_{|\bk|_1 > \xi } \big\|v_\bk f \big\|_{L_{q,w}(\RRd)}.			
	\end{equation}
	\hfill
\end{proof}

For given $1\le p,q \le \infty$ and $r \in \NN$ we make use of the abbreviations:
	\begin{equation}\nonumber
\lambda_n:=	\lambda_n(\bW^r_{p,w}(\RRd), L_{q,w}(\RRd)),  \ \ 		
d_n:=	d_n(\bW^r_{p,w}(\RRd), L_{q,w}(\RRd)).
\end{equation}

\begin{theorem} \label{thm:f-Ssf}
	Let $1\le p, q \le \infty$ and $r_{\lambda,p,q} >0$.  For every $n \in \NN$, let $\xi_n$ be the largest number such that $\operatorname{rank } \brac{\Vv_{\xi_n}} \le n$.
	 	Then we have that for $n \ge 2$,
	\begin{equation}\label{rho_n(W)}
	d_n \le \lambda_n
	\le 	\sup_{f \in \bW^r_{p,w}(\RRd)}\big\|f - \Vv_{\xi_n} f \big\|_{L_{q,w}(\RRd)} 
	\ll 
		\begin{cases}
		n^{-r_{\lambda}} (\log n)^{(r_{\lambda} + 1)(d-1)} &  \  \text{if} \ \ p = q, \\
		n^{-r_{\lambda,p,q}} (\log n)^{(r_{\lambda,p,q} + 1/q)(d-1)} &  \ \ \text{if} \ \ p \not= q<\infty,
		\\
		n^{-r_{\lambda,p,q}} (\log n)^{(r_{\lambda,p,q} + 1)(d-1)} &  \  \text{if} \ \ q=\infty.
	\end{cases}	
	\end{equation}
\end{theorem}

\begin{proof}  
To prove the upper bound  \eqref{rho_n(W)} we approximate a function $f \in \bW^r_{p,w}(\RRd)$
  by using the linear operator $\Vv_{\xi}$. Let us prove the case $p \not= q < \infty$ of \eqref{rho_n(W)}. The cases $p = q$ and $q=\infty $ can be proven in a similar manner. 
 
  From \eqref{rank Vv} it follows 
$$
2^{\xi_n} \xi_n^{d - 1} \asymp \operatorname{rank } \brac{\Vv_{\xi_n}} \asymp n.
$$ 
Hence we deduce the asymptotic equivalences
$$
\ 2^{-\xi_n}  \asymp n^{-1} (\log n)^{d-1}, \ \  \xi_n \asymp \log n,
$$
which together with Theorem \ref{thm:Vv_xi-error}  yield that
	\begin{equation*}\label{I_xi-error}
		\begin{aligned}
		d_n \le \lambda_n
		&\le 
	\sup_{f\in \bW^r_{p,w}(\RRd)}	
	\big\|f - \Vv_{\xi_n} f\big\|_{L_{q,w}(\RRd)}   
		\\
		&
		\leq 
		C 2^{- r_\lambda \xi_n} \xi_n^{(d-1)/q}
		\asymp  	n^{-r_{\lambda,p,q}} (\log n)^{(r_{\lambda,p,q} + 1/q)(d-1)}.
		\end{aligned}
	\end{equation*}
The upper bound in \eqref{rho_n(W)} for the case $p \not= q<\infty$ is proven.
	\hfill
\end{proof}

For $\bk \in \NNd$, the $d$-dimensional operators  $s_\bk$ are defined as the tensor  product of one-dimensional operators:
\begin{equation}\nonumber
	s_\bk:= \bigotimes_{i=1}^d s_{k_i}. 
\end{equation}
 For $\xi > 0$, we define the linear operator $\Ss_\xi$ for functions $f \in L_{2,w}(\RRd)$ by
\begin{equation}\nonumber
	\Ss_\xi f
	:= 	\sum_{|\bk|_1 \le \xi } s_{\bk} f.
\end{equation}	
Notice that the function $\Ss_\xi f$ belongs to the polynomial subspace 
\begin{equation}\nonumber
	\Pp_1(\xi) 
	:= 	\operatorname{span}\brab{p_\bs: \,  \bs \in H(\xi)},
\end{equation}	
where
\begin{equation}\nonumber
	H_1(\xi) 
	:= 	\bigcup_{\bk \in \NNd_0: \, |\bk|_1 \le \xi }\brab{\bs \in \NNd_0: \,  \bs \le 2^{\bk} }.
\end{equation}	
Notice that  by \eqref{|S_m f|<} $\Ss_{\xi}$ is a linear bounded operator  in $L_{q,w}(\RRd)$ for $4/3 < q < 4$, and
\begin{equation}\label{rank Vv}
	\operatorname{rank } \brac{\Ss_{\xi}} = |H_1(\xi) | 
	= \sum_{|\bk|_1 \le \xi} \prod_{j=1}^d (2^{k_j} - 1)
	\asymp 
	2^{\xi} \xi^{d - 1}.
\end{equation} 

In a way similar to the proofs of Lemma \ref{lemma:v_bk} and Theorem \ref{thm:f-Ssf} we can prove the following results.
\begin{lemma} \label{lemma:s_bk}
	Let $4/3< p < 4$, $1\le q \le \infty$ and $r_{\lambda,p,q} >0$.    Then we have that
	for every $ f \in W^r_{p,w}(\RRd)$,
	\begin{equation}\nonumber
		f
		= 	\sum_{\bk \in \NNd_0}s_\bk f 
	\end{equation}
	with absolute convergence in the space $L_{q,w}(\RRd)$  of the series, and
	\begin{equation}\nonumber
		\big\|s_\bk f\big\|_{L_{q,w}(\RRd)}  
		\leq C 2^{-r_{\lambda,p,q}|\bk|_1} \|f\|_{W^r_{p,w}(\RRd)}, 
		\ \  \bk \in \NNd_0.
	\end{equation}
\end{lemma}

\begin{theorem} \label{thm:Ss_xi-error}
	Let $4/3< p < 4$,  $1\le q \le \infty$ and $r_{\lambda,p,q} >0$.   For every $n \in \NN$, let $\xi_n$ be the largest number such that $\operatorname{rank} \brac{\Ss_{\xi_n}} \le n$.
	Then we have that for $n \ge 2$,
	\begin{equation}\nonumber
		d_n \le \lambda_n
		\le 	\sup_{f \in \bW^r_{p,w}(\RRd)}\big\|f - \Ss_{\xi_n} f \big\|_{L_{q,w}(\RRd)} 
		\ll 
		\begin{cases}
			n^{-r_{\lambda}} (\log n)^{(r_{\lambda} + 1)(d-1)} &  \  \text{if} \ \ p = q, \\
			n^{-r_{\lambda,p,q}} (\log n)^{(r_{\lambda,p,q} + 1/q)(d-1)} &  \ \ \text{if} \ \ p \not= q<\infty,
			\\
			n^{-r_{\lambda,p,q}} (\log n)^{(r_{\lambda,p,q} + 1)(d-1)} &  \  \text{if} \ \ q=\infty.
		\end{cases}	
	\end{equation}
\end{theorem}

\section{Right convergence rate of $n$-widths}
\label{Asymptotic order of n-widths}
In this section, we  prove   the  right convergence rate of  $\lambda_n\big(\bW^r_{2,w}(\RRd), L_{2,w}(\RRd)\big)$ and   $d_n\big(\bW^r_{2,w}(\RRd), L_{2,w}(\RRd)\big)$ in the case when   the generating weight $w$ is given as in \eqref{w(x)} with $\lambda = 2, 4$.  

For $r \in \NN$ and $\bk \in \NNd_0$, we define
\begin{equation}\nonumber
	\rho_{\lambda,r,\bk}: = \prod_{j=1}^d \brac{k_j + 1}^{r_\lambda},
\end{equation}
where recall, $r_\lambda$ is given as in \eqref{r_lambda} and $\lambda$ as in \eqref{w(x)}.
Denote by $\Hh^{r_\lambda}_w(\RRd)$ the space of all   functions $f \in L_{2,w}(\RRd)$ represented by the  series \eqref{ONP-series} for which  the norm
\begin{equation}\nonumber
	\norm{f}{\Hh^{r_\lambda}_w(\RRd)} := \brac{\sum_{\bk \in \NNd_0} |\rho_{\lambda, r,\bk}\hat{f}(\bk)|^2}^{1/2}
\end{equation}
is finite.

For  functions $f \in \Hh^{r_\lambda}_w(\RRd)$, we construct a hyperbolic cross polynomial approximation based on truncations of the orthonormal polynomial series \eqref{ONP-series}. For the hyperbolic cross
$$G(\xi):= \brab{\bk \in \NNd_0: \rho_{\lambda,r,\bk} \le \xi}, \ \ \xi  \ge 1,$$
the truncation $S_\xi^*(f)$ of the  series \eqref{ONP-series} on this set is defined by
\begin{equation*}\label{S_xi}
S_\xi^*(f)	:= \sum_{\bk \in G(\xi)} \hat{f}(\bk) p_\bk.
\end{equation*}
Notice that $S_\xi^*$ is a linear projection from $L_2(\RRd,\gamma)$ onto the linear subspace $L(\xi)$ spanned by the orthonormal polynomials $p_\bk$, $\bk \in G(\xi)$, and $\dim L(\xi) = |G(\xi)|$.

We will need the following Tikhomirov lemma which is often used for lower estimation of Kolmogorov $n$-widths \cite[Theorem 1]{Tikh1960}. 
\begin{lemma} \label{lemma: Tikhomirov} 
If $X$ is a Banach space and $U$ the ball of radius $\rho >0$ in a linear $n+1$-dimensional subspace of $X$, then 
$$
d_n(U,X)=\rho.
$$
\end{lemma}

\begin{theorem} 	\label{theorem:widths:p=q=2}
We can construct a sequence $\brab{\xi_n}_{n=2}^\infty$ with $|G(\xi_n)|\le n$ so that for $n \ge 2$,
\begin{equation}
\begin{aligned}\label{widths:p=q=2}
	\sup_{f\in \boldsymbol{\Hh}^{r_\lambda}_w(\RRd)} \norm{f - S_{\xi_n}^*(f)}{L_{2,w}(\RRd)} 
	&\asymp
	\lambda_n(\boldsymbol{\Hh}^{r_\lambda}_w(\RRd), L_{2,w}(\RRd))\\
	&\asymp 
	d_n(\boldsymbol{\Hh}^{r_\lambda}_w(\RRd), L_{2,w}(\RRd))
	\asymp n^{-r_\lambda} (\log n)^{r_\lambda (d - 1)}.
\end{aligned}		
	\end{equation}
\end{theorem}
\begin{proof}
Since 	$L_{2,w}(\RRd)$ is a Hilbert space, we have the equality 
$$
\lambda_n(\boldsymbol{\Hh}^{r_\lambda}_w(\RRd), L_{2,w}(\RRd))
= 
d_n(\boldsymbol{\Hh}^{r_\lambda}_w(\RRd), L_{2,w}(\RRd)).
$$
To prove the upper bounds in \eqref{widths:p=q=2} it is sufficient to construct a sequence $\brab{\xi_n}_{n=2}^\infty$ so that $|G(\xi_n)|\le n$ and
\begin{align}\label{f-S_xi(f)}
	\sup_{f\in \boldsymbol{\Hh}^{r_\lambda}_w(\RRd)} \norm{f - S_{\xi_n}^*(f)}{L_{2,w}(\RRd)}
	\ll
n^{-r_\lambda} (\log n)^{r_\lambda (d - 1)}.
\end{align}		
From Parseval's identity \eqref{P-id}  we have that for every $f \in \boldsymbol{\Hh}^{r_\lambda}_w(\RRd)$ and $\xi > 1$,
\begin{equation} \label{f-S_xi(f)2}
		\norm{f - S_{\xi}^*(f)}{L_{2,w}(\RRd)}
		\ll      
 \xi^{- r_\lambda}.
\end{equation}
Indeed,
\begin{equation} \nonumber
\begin{aligned}
	\norm{f - S_{\xi}^*(f)}{L_{2,w}(\RRd)}^2
	&=
	\sum_{\bk \notin G(\xi)} \hat{f}(\bk)^2
	\ll
	\xi^{- 2r_\lambda}	\sum_{\bk \notin G(\xi)} |\rho_{\lambda,r,\bk}\hat{f}(\bk)|^2
	\\&
		\ll      
	\xi^{- 2r_\lambda} \norm{f}{\Hh^{r_\lambda}_w(\RRd)} \le \xi^{- 2r_\lambda} .
\end{aligned}
\end{equation}
Let 	$\brab{\xi_n}_{n=2}^\infty$ be the sequence of $\xi_n$ defined as the largest number satisfying the condition 	$|G(\xi_n)|\le n$. From the relation $$|G(\xi_n)| \asymp \xi_n (\log \xi_n)^{d-1},$$
 see, e.g., \cite[page 130]{Tem93B}, we derive that 
 $$\xi_n^{-r_\lambda} \asymp n^{-r_\lambda} (\log n)^{r_\lambda(d-1)}$$
  which together with \eqref{f-S_xi(f)2} yields \eqref{f-S_xi(f)}.

To show the lower bounds of \eqref{widths:p=q=2} we will apply Lemma \ref{lemma: Tikhomirov}.  If  
$$
U(\xi):=\brab{f \in L(\xi): \norm{f }{L_{2,w}(\RRd)} \le 1}
$$ 
and $f \in U(\xi)$, then by Parseval's identity \eqref{P-id} and the definition of $\boldsymbol{\Hh}^{r_\lambda}_w(\RRd)$, similarly to \eqref{f-S_xi(f)2}, we deduce that
$$\norm{f}{\boldsymbol{\Hh}^{r_\lambda}_w(\RRd)}	\ll \xi^{r_\lambda}.$$
  This means that  
$$C \xi^{-r_\lambda} U(\xi) \subset \boldsymbol{\Hh}^{r_\lambda}_w(\RRd)$$
 for some $C > 0$. Let 	$\brab{\xi'_n}_{n=2}^\infty$ be the sequence of $\xi'_n$ defined as the smallest number satisfying the condition 	$$|G(\xi'_n)|\ge n + 1.$$
  Then $$\dim L(\xi'_n) = |G(\xi'_n)|\ge n + 1,$$
   and
similarly as in the upper estimation,  
$$(\xi'_n)^{-r_\lambda} \asymp n^{-r_\lambda} (\log n)^{(d-1)r_\lambda}.$$ 
By Lemma \ref{lemma: Tikhomirov} for the smallest quantity 
$d_n$ in \eqref{widths:p=q=2} we have that
\begin{align*}
d_n(\boldsymbol{\Hh}^{r_\lambda}_w(\RRd), L_{2,w}(\RRd))
	\ge
	d_n(C(\xi'_n)^{-r_\lambda} U(\xi'_{n+1}), L_{2,w}(\RRd))
	=
	C(\xi'_n)^{-r_\lambda}
	\asymp 
	n^{-r_\lambda} (\log n)^{r_\lambda(d-1)} .
\end{align*}

\hfill	
\end{proof}

\begin{theorem}\label{thm:norm-eq}
Let $\lambda = 2, 4$  for the generating univariate weight $w$  as in \eqref{w(x)}.  Then we have the norm equivalence
	\begin{equation}\label{norm-eq}
		\norm{f}{W^r_{2,w}(\RRd)}
		\asymp 
		\norm{f}{\Hh^{r_\lambda}_w(\RRd)}, \ \  f \in W^r_{2,w}(\RRd).
	\end{equation}
\end{theorem}

This theorem  has been proven  in \cite[Lemma 3.4]{DK2022} for $\lambda = 2$, and in \cite{DD2024} for $\lambda = 4$.

Due to the norm equivalence \eqref{norm-eq} in Theorem \ref{thm:norm-eq}, we identify $W^r_{2,w}(\RRd)$ with $\Hh^{r_\lambda}_w(\RRd)$ for the cases  $\lambda = 2, 4$ and $r \in \NN$.
In these cases, Theorem \ref{theorem:widths:p=q=2} gives the following result on right asymptotic order of  linear $n$-widths  $\lambda_n\big(\bW^r_{2,w}(\RRd), L_{2,w}(\RRd)\big)$ and   Kolmogorov $n$-widths $d_n\big(\bW^r_{2,w}(\RRd), L_{2,w}(\RRd)\big)$.

\begin{theorem} 	\label{thm:widths:p=q=2W}
Let $\lambda =2, 4$. We have for $n \ge 2$,
\begin{align}\label{W-sampling-widths:p=q=2}
	\lambda_n(\bW^r_{2,w}(\RRd), L_{2,w}(\RRd)) 
	=
		d_n(\bW^r_{2,w}(\RRd), L_{2,w}(\RRd)) 
	\asymp 
	n^{- r_\lambda} (\log n)^{r_\lambda(d-1)}.
\end{align}		
\end{theorem}

We conjecture that the right convergence rate \eqref{W-sampling-widths:p=q=2} is still holds true at least  for every even $\lambda$.

\medskip
\noindent
{\bf Acknowledgments:}  
This work is funded by the Vietnam National Foundation for Science and Technology Development (NAFOSTED) in the frame of the Vietnamese-Swiss Joint Research Project 
under  Grant IZVSZ2$_{ - }$229568. 
A part of this work was done when  the author was working at the Vietnam Institute for Advanced Study in Mathematics (VIASM). He would like to thank  the VIASM  for providing a fruitful research environment and working condition. 
	
\bibliographystyle{abbrv}
\bibliography{WeightedApproximation}
\end{document}